\documentclass[preprint,10pt]{elsarticle}
\usepackage{etoolbox}
\AtBeginEnvironment{align}{
\setlength\abovedisplayskip{2pt}
\setlength\belowdisplayskip{2pt}
}

\usepackage{xstring}
 \usepackage{float}
\usepackage{amssymb}
\usepackage{algorithm}
\usepackage{epstopdf}
\usepackage{amsfonts}
\usepackage{amsmath}
\usepackage[colorlinks,
            linkcolor=red,
            anchorcolor=blue,
            citecolor=green
            ]{hyperref}
\usepackage{extarrows}
\usepackage{multirow}
\usepackage{graphicx}
\usepackage{subfigure}

\usepackage{mathptmx}
\usepackage{enumitem}  
\setlist[itemize]{
    topsep=2pt,      
    partopsep=0pt,
    itemsep=2pt,     
    parsep=0pt,      
}
\usepackage{titlesec}
\titlespacing{\section}{0pt}{2ex plus 1ex minus .2ex}{1ex plus .2ex}
\titlespacing{\subsection}{0pt}{1ex plus 0.8ex minus .2ex}{1ex plus .2ex}
\biboptions{numbers,sort&compress}
\numberwithin{equation}{section}
\journal{Applied Mathematics Letters}

\begin{document}
\newtheorem{Thm}{Theorem}[section]
\newtheorem{Cor}[Thm]{Corollary}
\newtheorem{Lem}[Thm]{Lemma}
\newtheorem{Prop}[Thm]{Proposition}
\newtheorem{Def}[Thm]{Definition}
\newtheorem{Rem}[Thm]{Remark}
\newtheorem{Assump}[Thm]{Assumption}

\newtheorem{Exam}[Thm]{Example}

\begin{frontmatter}

\title{Sparse descent subspace-Newton  thresholding  algorithm for  sparse optimization problems}

\author[a]{Zhong-Feng Sun}
\ead{zfsun@sdut.edu.cn}

\author[b]{Nan Meng}
\ead{Nan.Meng@nottingham.edu.cn}

\author[c]{Yun-Bin Zhao\corref{cor}}
\ead{yunbinzhao@cuhk.edu.cn}
\cortext[cor]{Corresponding author. }

\address[a]{School of Mathematics and Statistics, Shandong University of Technology, Zibo, Shandong,  China}

\address[b]{School of Mathematical Sciences, University of Nottingham Ningbo China, Ningbo, Zhejiang, China}

\address[c]{Shenzhen International Center for Industrial and Applied
Mathematics, SRIBD, The Chinese University of Hong Kong, Shenzhen, China}

\begin{abstract}
In this paper, we propose a sparse descent direction derived from the Newton direction within a low-dimensional subspace for sparse optimization problems.  Building upon this direction and incorporating hard thresholding and orthogonal projection, we develop a novel algorithm  termed the sparse descent subspace-Newton  thresholding (SDSNT) algorithm. We further  establish its global convergence under the restricted isometry property $\delta_{3k}\leq \sqrt{3/20}$ of the measurement matrix. Numerical experiments on random sparse vector recovery and audio signal reconstruction reveal that  SDSNT achieves comparable performance to several state-of-the-art algorithms. Notably, relative to most of them, it delivers superior recovery at high sparsity levels and improved efficiency at low sparsity levels.
\end{abstract}

\end{frontmatter}

 \section{Introduction}
  In this paper, we aim to recover a sparse vector $x\in \mathbb{R}^n$ from its linear measurements $y= Ax+\nu \in \mathbb{R}^m$ with unknown  errors $\nu \in \mathbb{R}^m$, where $ A \in\mathbb{R}^{m\times n}$ is a given measurement matrix with $m \ll n.$ The recovery of $x$ from the data $(y,A)$  can be formulated as the following sparse optimization problem
 \begin{align}\label{model}
\arg \min_{z\in \mathbb{R}^n} \left\{\frac{1}{2} \| y-Az\|_2^2: ~ \| z\|_0\leq k\right\},
\end{align}
where $\| \cdot\|_0$ denotes the sparsity level (i.e., the number of nonzero entries) of a vector, and $k$ is an  estimate of $\|x\|_0$.

 Hard thresholding algorithms constitute a mainstream class of sparse optimization methods for solving \eqref{model}. Gradient descent and Newton's method are two fundamental approaches to unconstrained optimization. By combining gradient descent with hard thresholding, the iterative hard thresholding (IHT) algorithm \cite{BD08} was developed for solving  \eqref{model}. While IHT enjoys an extremely low computational cost per iteration, it suffers from slow convergence and high sensitivity to noise. To overcome these drawbacks, the hard thresholding pursuit (HTP)  algorithm \cite{F11} was proposed   by incorporating  orthogonal projection into  IHT. Among its variants, the heavy-ball-based HTP algorithm was devised by fusing the gradient and momentum directions \cite{SZZM23}. While gradient descent is sensitive to poor scaling, Newton's method is largely unaffected by it and usually exhibits faster local convergence \cite{NW06}. By integrating  Newton's method with hard thresholding, the Newton-step-based HTP algorithm was developed in \cite{MZ20}. Since this algorithm requires computing the inverse of the modified Hessian $(A^TA+\varepsilon I)^{-1}$, where the Hessian $A^TA$ is singular, $I$ is the identity matrix and $\varepsilon>0$, its computational cost becomes prohibitive for high-dimensional problems. To reduce the computational cost and circumvent the Hessian singularity, the compressed Newton-direction-based thresholding pursuit algorithm was introduced in \cite{MZ26}, which projects the  Newton step onto a lower-dimensional subspace and subsequently maps it back to the original space via diagonal regularization.   Furthermore, the combination of the restricted Newton and gradient directions was utilized  to design the  Newton HTP (NHTP)  algorithm \cite{ZXD21}. An analogous method has also been developed for sparse optimization problems with box constraints \cite{YL26}.

Additionally, full gradient descent coincides with partial gradient descent under certain conditions for sparse optimization.  In what follows, we take  HTP  as an example to illustrate this. HTP  is briefly summarized as follows:
\begin{align}
S^{p+1}&=\mathcal{L}_k(x^p+ d^p),\label{alg-HTP1}\\
x^{p+1}&=\arg \min_{z\in \mathbb{R}^n} \left\{ \frac{1}{2} {\|y-Az\|}_2^2: \mathrm{supp}(z)\subseteq S^{p+1}\right\},\label{alg-HTP2}
\end{align}
where $d^p$ is the negative gradient  at the current point $x^p$, and $\mathcal{L}_q(\cdot) (q\geq 1)$ is the index set for the  $q$ largest absolute entries of  a vector. Moreover, the supports of $x^p$ and $d^p$ are disjoint due to the first-order optimality condition of \eqref{alg-HTP2}. Thus, \eqref{alg-HTP1} is equivalent to $S^{p+1}=\mathcal{L}_k(x^p+ (d^p)_{\Omega_q})$ for any integers $q\in[k,n)$, where $\Omega_q=\mathcal{L}_q(d^p)$ and $(d^p)_{\Omega_q}$ is a $q$-sparse vector which retains the elements of $d^p$ with indices lying in ${\Omega_q}$ and sets  the others to zero. That is, the partial gradient (or sparse descent direction) can serve as a replacement for its full-gradient counterpart    in HTP. Inspired by this observation, we integrate the  sparse descent direction with the subspace Newton method to  propose the sparse descent subspace-Newton  thresholding (SDSNT) algorithm. We establish the global convergence of SDSNT under the restricted isometry property $\delta_{3k}\leq \sqrt{3/20}$. Moreover,  numerical results on synthetic and real data reveal that  SDSNT achieves comparable performance to several algorithms, including NHTP \cite{ZXD21}, HTP \cite{F11}, natural thresholding  pursuit (NTP)  \cite{ZL22},  subspace pursuit (SP) \cite{DM09},  and compressive sampling matching pursuit (CoSaMP) \cite{NT09} algorithms.

The paper is organized as follows. In Section \ref{sec-pre-alg}, we introduce some notation, definition, basic inequalities and the proposed algorithm. In Section \ref{sec-glob-converg}, the global convergence of SDSNT  is rigorously proved. In Section \ref{sec-exper}, numerical experiments on synthetic data and audio signal data are conducted. The conclusion is summarized in the final section.

\section{Preliminary and algorithm} \label{sec-pre-alg}

\subsection{Preliminary} \label{subsec-pre}
Let $N$ be the index set $\{1, \dots, n\}$. For  any set $ S\subseteq N,$ we write $|S|$  for the cardinality of $ S$, and $\overline{S} =N\setminus S $  for the complement of $ S$ in $N$.  Given $ x =(x_1, \dots, x_n)^T \in \mathbb{R}^n$, $\textrm{supp}(x)=\{i : x_i\neq 0\}$ is referred to as  the support of $x.$ Let $\mathcal{L}_s(x) $ be the index set for the  $s$ largest absolute entries of $x. $  If $|\textrm{supp}(x)|\leq s,$ then $x$ is said to be $s$-sparse. For $ x\in \mathbb{R}^n $ and $ S\subseteq N,$   we use  $x_{S} $ to denote two related vectors, depending on context: 1.  the masked vector in  $\mathbb{R}^n$ that preserves  $ x_i $ for $i \in S$ and sets $ x_i=0$ for $i\notin S;$ 2. the compressed  vector in  $\mathbb{R}^{|S|}$ formed by extracting only the entries $ x_i $ with $i \in S$ and discarding the rest.
 We write $A_S$ for the submatrix of $A \in \mathbb{R}^{m\times n}$ consisting of columns indexed by $S.$ Let $I$ denote the identity matrix, whose dimension can be determined from the context.

The following definition was originally proposed in \cite[definition 1.1] {CT05} and is essential for analyzing our algorithm.
\begin{Def}\cite{CT05} \label{Def-RIP-2005} Given a matrix $A\in \mathbb{R}^{m\times n}$ with $m\ll n$. If there exists a smallest constant   $ 0\leq \delta_s<1$ such that
 \begin{align}
 (1-\delta_s) \|x\|_2^2 \leq \|Ax\|_2^2\leq (1+\delta_s) \|x\|_2^2  \nonumber
  \end{align}
for any  $ s$-sparse vector $ x \in \mathbb{R}^n,$  the matrix $A$ is said to satisfy the restricted isometry property (RIP)  of order $s$. Moreover, $ \delta_s$ is known as the restricted isometry constant (RIC) of order $s$, and it   admits an equivalent form \cite{F11}:
 \begin{align}
  \delta_s=\max_{S\subseteq N,~|S|\leq s}\|A^T_SA_S-I\|_2. \nonumber
  \end{align}
\end{Def}

Two basic inequalities given in the following lemma will be frequently used in the next section.
\begin{Lem} \emph{\cite{F11}} \label{lem-basic-ineq}
Let $u \in \mathbb{R}^n$ and $ v \in \mathbb{R}^m$ be two vectors, $s \in N$ be a positive  integer and $W\subseteq N$ be an index set.
\begin{itemize}
\item [\rm (i)]   If  $ |W\cup  \mathrm{supp} (u)|\leq s$, then
$\| \left ( (I-A^ T A)u\right)_W\|_2\leq\delta_s {\| u \|}_2.$
\item [\rm (ii)]   If  $ |W|\leq s$, then
$\| \left ( A^ Tv\right)_W\|_2\leq\sqrt{1+\delta_s} {\| v \|}_2.$
\end{itemize}
\end{Lem}

\subsection{Algorithm} \label{subsec-alg}

From Definition \ref{Def-RIP-2005}, under the RIP of order $s$, the  principal submatrix $A^T_{S}A_{S}$ of the Hessian is positive definiteness for any  $ S\subseteq N$ with $|S|= s$. Accordingly, we construct a sparse descent direction $d\in \mathbb{R}^n$ as follows:  $d_{\overline{S}}=0$ and $d_S=(A^T_{S}A_{S})^{-1}A_S^ T(y-Ax)$, with  $d_S$ being the Newton direction in the $s$-dimensional coordinate subspace associated with $S$.  It can be regarded as a special case of the compressed Newton search direction $d_{\mathrm{CN}}$ with $\gamma=0$ in \cite{MZ26}; however,  both the algorithmic framework and its convergence analysis in \cite{MZ26} were  established only for $\gamma>0$.
By embedding the sparse direction $d$ into the hard thresholding framework, we propose the SDSNT algorithm for sparse optimization problems.

\begin{algorithm}\label{alg-SDSNT}
\caption{Sparse descent subspace-Newton  thresholding  (SDSNT) algorithm}
Input $(A,y,k)$, stepsize $\alpha> 0$ and  initial point $x^0=0 $.
\begin{itemize}
\item[ S1.] Let  $r^p=A^ T(y-Ax^p)$,  $\Omega_k =\mathcal{L}_k(r^p),$ and $M^p =(A^T_{\Omega_k}A_{\Omega_k})^{-1}.$ Set
 \begin{align}\label{alg-direction}
u^p=x^p+\alpha\cdot d^p,
\end{align}
where
 \begin{align}\label{def-dp}
(d^p)_{\Omega_k}=M^p\cdot(r^p)_{\Omega_k},~~(d^p)_{\overline{\Omega_k}}=0.
\end{align}
\item[S2.] Let $S^{p+1}=\mathcal{L}_k(u^p)$, and
 \begin{align}\label{alg-pursuit}
x^{p+1}=\arg \min_{z\in \mathbb{R}^n} \left\{ \frac{1}{2}  {\| y-Az\|}_2^2: \mathrm{supp}(z)\subseteq S^{p+1}\right\}.
\end{align}
The above steps continue to be  looped unless a stopping criterion is reached.
\end{itemize}
\end{algorithm}

\section{Global convergence}\label{sec-glob-converg}
The following lemma can be obtained by Eq. (3.21) in \cite{F11} and Eq. (3.17) in \cite{SZZM23} directly.
\begin{Lem}\emph{\cite{F11,SZZM23}}  \label{Lem-projection-1}
 Given  vectors $x,~u^p\in \mathbb{R}^n$, the index sets $ S=\mathcal{L}_k(x)$ and $S^{p+1}=\mathcal{L}_k(u^p),$ and  the noisy measurements $y = Ax+ \nu\in \mathbb{R}^m$ of $x$, where $\nu$  is the noise vector. The iterate $ x^{p+1}$ generated by  \eqref{alg-pursuit}
obeys
\begin{align}
\left\lVert x^{p+1}-x_S\right\rVert_2
\leq \frac{1}{\sqrt{1-\delta^2_{2k}}}\left\lVert (x^{p+1}-x_S)_{\overline{S^{p+1}}}\right\rVert_2+\frac{ \sqrt{1+\delta_{k}}}{1-\delta_{2k}}\left\lVert \nu' \right\rVert_2,\nonumber
\end{align}
where $\nu'=\nu+Ax_{\overline{S}}$, and
\begin{align}
\|(x^{p+1}-x_S)_{\overline{ S^{p+1}}}\|_2\leq  \sqrt{2}\|(u^p-x_S)_{S^{p+1}\bigtriangleup S}\|_2,\nonumber
\end{align}
where  $S^{p+1}\bigtriangleup S:=(S^{p+1}\setminus S)\cup(S\setminus S^{p+1})$ is the symmetric difference of $S^{p+1}$ and $S.$
\end{Lem}

The error bound for SDSNT  is established in the following theorem.
 \begin{Thm}\label{thm-error-bound}
 Suppose that the RIC, $\delta_{3k},$ of the matrix $A$ and the stepsize $\alpha$ satisfy that
 \begin{align}\label{RIC-bound}
\delta_{3k}\leq \sqrt{\frac{3}{20}},~~~1-\frac{C_1-C_3}{1+\sqrt{2}\delta_{3k}-C_3}<\alpha<1+\frac{C_1-C_3}{\sqrt{2}\delta_{3k}C_2+1},
\end{align}
where the constants $C_i, i=1,2,3$ are given as follows,
 \begin{align}\label{Constants}
C_1=\sqrt{\frac{1-\delta^2_{3k}}{2}},~~~C_2=\frac{\delta_{3k}}{1-\delta_{3k}},~~~C_3=\sqrt{2}\delta_{3k}\sqrt{1+C_2^2},
\end{align}
and $C_3<C_1<1$. Let $y=Ax+\nu$ be the measurements of $ x$ with errors $ \nu . $
Then the iterates $\{ x^p\}$ generated by SDSNT  satisfies
\begin{align}\label{error-bound}
\|x^{p+1} - x_S \|_2 \leq \rho\|x^p - x_S \|_2+C\|\nu'\|_2,
\end{align}
where  $S=\mathcal{L}_k(x)$, $\nu'=\nu+Ax_{\overline{S}}$, and $\rho,~C$ are given  as
\begin{align}\label{rho-C}
\rho=\frac{\sqrt{2}}{\sqrt{1-\delta^2_{2k}}}\left(|1-\alpha|+\sqrt{2}\delta_{3k}C_4\right)<1,
~~~~C=\frac{1}{\sqrt{1-\delta_{3k}}}\cdot\left(2C_4 +\frac{\sqrt{2}\alpha}{1-\delta_k} \right) +\frac{\sqrt{1+\delta_{k}}}{1-\delta_{2k}},
\end{align}
where $C_4=\sqrt{\left(\frac{\alpha\delta_{3k}}{1-\delta_k}\right)^2+1}.$
\end{Thm}
\emph{Proof.}  Denote $Q^p:=\mathrm{supp}(x^p)$ and $W:=S\cup Q^p\cup\Omega_k$, where $S=\mathcal{L}_k(x)$, $\Omega_k =\mathcal{L}_k(r^p),$ $|Q^p|\leq k$ and $|S|=|\Omega_k|=k$. From SDSNT algorithm, we have
$S^{p+1}\subseteq\mathrm{supp}(u^p)\subseteq Q^p\cup \Omega_k,$ which in turn implies
 $(S^{p+1}\setminus S)\cup(S\setminus S^{p+1})\subseteq S^{p+1}\cup S \subseteq W$.
 It follows from Lemma \ref{Lem-projection-1} that
 \begin{align} \label{xp1-x-up-xnorm-2}
\|x^{p+1} - x_S \|_2
\leq\frac{\sqrt{2}}{\sqrt{1-\delta^2_{2k}}}\|(u^p  - x_S )_{W}\|_2+\frac{ \sqrt{1+\delta_{k}}}{1-\delta_{2k}}\left\| \nu' \right\|_2.
\end{align}
Next, we aim to estimate the upper bound of  $\|(u^p  - x_S )_{W}\|_2$ by $\|x^p-x_S\|_2$ and ${\| \nu'\|}_2$. Note that $W$ can be split into two disjoint sets $(S\cup Q^p)\setminus\Omega_k$ and $\Omega_k$.

{\bf Firstly, we estimate the upper bound of $\|(u^p-x_S)_{(S\cup Q^p)\setminus\Omega_k}\|_2$.} From the orthogonality of \eqref{alg-pursuit}, we see that $Q^p\cap\mathrm{supp}(r^p)=\emptyset$, which implies that $(r^p )_{Q^p}=0$ and $(x^p )_{\Omega_k}=0$. From \eqref{alg-direction} and \eqref{def-dp}, we obtain
\begin{align} \label{up-x_COk}
(u^p - x_S )_{(S\cup Q^p)\setminus\Omega_k}&=(x^p - x_S )_{(S\cup Q^p)\setminus\Omega_k}\nonumber \\
&=(x^p - x_S+r^p )_{(S\cup Q^p)\setminus\Omega_k}-(r^p )_{(S\cup Q^p)\setminus\Omega_k}\nonumber \\
&=(x^p - x_S+r^p )_{(S\cup Q^p)\setminus\Omega_k}-(r^p )_{S\setminus\Omega_k},
\end{align}
where the last equality is given by $(r^p )_{Q^p}=0$. Since $\Omega_k =\mathcal{L}_k(r^p)$ and  $|S|=|\Omega_k|=k$, we get
\begin{align} \label{rp_COk-norm}
\|(r^p )_{S\setminus\Omega_k}\|_2\leq \|(r^p )_{\Omega_k\setminus S}\|_2=\|(x^p-x_S+r^p )_{\Omega_k\setminus S}\|_2,
\end{align}
where the equality results from   $(x^p )_{\Omega_k\setminus S}=0$ and $(x_S )_{\Omega_k\setminus S}=0$. Merging \eqref{up-x_COk} and \eqref{rp_COk-norm} leads to
\begin{align} \label{up-x_COk-norm-1}
\|(u^p - x_S )_{(S\cup Q^p)\setminus\Omega_k}\|_2
&\leq\|(x^p - x_S+r^p )_{(S\cup Q^p)\setminus\Omega_k}\|_2+\|(x^p-x_S+r^p )_{\Omega_k\setminus S}\|_2\nonumber \\
&\leq\sqrt{2}\sqrt{\|(x^p - x_S+r^p )_{(S\cup Q^p)\setminus\Omega_k}\|_2^2+\|(x^p-x_S+r^p )_{\Omega_k\setminus S}\|_2^2}\nonumber \\
&\leq\sqrt{2}\|(x^p - x_S+r^p )_{W}\|_2,
\end{align}
where the last inequality results from $((S\cup Q^p)\setminus\Omega_k)\cup (\Omega_k\setminus S)\subseteq W$ and $((S\cup Q^p)\setminus\Omega_k)\cap (\Omega_k\setminus S)=\emptyset$.
Since $y=Ax+\nu$, we have
\begin{align}\label{rp}
r^p=A^ T(y-Ax^p)=A^T\left[A(x_S-x^p)+\nu'\right],
\end{align}
where $\nu'=\nu+Ax_{\overline{S}}$.
Since $\mathrm{supp}(x^p - x_S)\subseteq W$ and $|W|\leq 3k$,
from \eqref{rp} and Lemma \ref{lem-basic-ineq}, we obtain
\begin{align} \label{xp-x+rp_norm}
\|(x^p - x_S+r^p)_W\|_2&=\|[(I-A^TA)(x^p-x_S)+A^T\nu']_W\|_2\nonumber \\
&\leq\|[(I-A^TA)(x^p-x_S)]_W\|_2+\|(A^T\nu')_W\|_2\nonumber \\
&\leq \delta_{3k}\|x^p-x_S\|_2+\sqrt{1+\delta_{3k}} {\| \nu'\|}_2.
\end{align}
Combing \eqref{up-x_COk-norm-1}  and \eqref{xp-x+rp_norm}, we get
\begin{align} \label{up-x_COk-norm-2}
\|(u^p - x_S )_{(S\cup Q^p)\setminus\Omega_k}\|_2
\leq \sqrt{2}\delta_{3k}\|x^p-x_S\|_2+\sqrt{2(1+\delta_{3k})} \| \nu'\|_2.
\end{align}

{\bf Secondly, we estimate the upper bound of $\| (u^p - x_S )_{\Omega_k}\|_2$.}
  From \eqref{alg-direction}, \eqref{def-dp} and \eqref{rp}, we obtain
\begin{align} \label{up-x_Ok}
(u^p - x_S )_{\Omega_k}=&(x^p - x_S )_{\Omega_k}+\alpha M^p(r^p)_{\Omega_k}\nonumber \\
 =&(x^p - x_S )_{\Omega_k}+\alpha M^p A^T_{\Omega_k}\left[A(x_S-x^p)+\nu'\right]\nonumber \\
  =&(x^p - x_S )_{\Omega_k}+\alpha M^p A^T_{\Omega_k}[A_{(S\cup Q^p)\setminus\Omega_k}(x_S-x^p)_{(S\cup Q^p)\setminus\Omega_k}+A_{\Omega_k}(x_S-x^p)_{\Omega_k}+\nu']\nonumber \\
    =&(1-\alpha)(x^p - x_S )_{\Omega_k}+\alpha M^p[A^T_{\Omega_k}A_{(S\cup Q^p)\setminus\Omega_k}(x_S-x^p)_{(S\cup Q^p)\setminus\Omega_k}+A^T_{\Omega_k}\nu'],
\end{align}
where the third equality is given by $\mathrm{supp}(x_S-x^p)\subseteq W$ and $W=[(S\cup Q^p)\setminus\Omega_k]\cup \Omega_k$.
Since $\Omega_k\cap[(S\cup Q^p)\setminus\Omega_k]=\emptyset$, thus $A^T_{\Omega_k}A_{(S\cup Q^p)\setminus\Omega_k}$ is a submatrix of $A^T_{W}A_{W}-I$. Hence,
\begin{align}\label{inq-submartrix-norm}
\|A^T_{\Omega_k}A_{(S\cup Q^p)\setminus\Omega_k}\|_2\leq\|A^T_{W}A_{W}-I\|_2\leq \delta_{3k},
\end{align}
where the last inequality results from Definition \ref{Def-RIP-2005}  with $|W|\leq 3k$. From Definition \ref{Def-RIP-2005},  the eigenvalues of  $A^T_{\Omega_k}A_{\Omega_k}$ lie in  $[1-\delta_k,1+\delta_k]$,  and thus those of  $M^p=(A^T_{\Omega_k}A_{\Omega_k})^{-1}$ belong to  $[ \frac{1}{1+\delta_k}, \frac{1}{1-\delta_k}]$. Hence,
$\|M^p\|_2\leq \frac{1}{1-\delta_k}.$
Combining this bound with \eqref{up-x_Ok} and \eqref{inq-submartrix-norm}, we arrive at
\begin{align}\label{up-x_Ok-norm}
\| (u^p - x_S )_{\Omega_k}\|_2
 &\leq|1-\alpha|\cdot\|(x^p - x_S )_{\Omega_k}\|_2+\alpha\|M^p\|_2(\|A^T_{\Omega_k}A_{(S\cup Q^p)\setminus\Omega_k}\|_2\cdot\|(x^p-x_S)_{(S\cup Q^p)\setminus\Omega_k}\|_2+\|A^T_{\Omega_k}\nu'\|_2)\nonumber \\
&\leq |1-\alpha|\cdot\|(x^p - x_S )_{\Omega_k}\|_2+ \frac{\alpha}{1-\delta_k}\left(\delta_{3k}\|(x^p-x_S)_{(S\cup Q^p)\setminus\Omega_k}\|_2+\|(A^T\nu')_{\Omega_k}\|_2\right)\nonumber\\
&\leq |1-\alpha|\cdot\|x^p - x_S \|_2+ \frac{\alpha}{1-\delta_k}\left(\delta_{3k}\|(u^p-x_S)_{(S\cup Q^p)\setminus\Omega_k}\|_2+\sqrt{1+\delta_{3k}} {\| \nu'\|}_2\right),
\end{align}
where the last inequality is obtained by the first equality of \eqref{up-x_COk} and Lemma \ref{lem-basic-ineq} with $|\Omega_k|<3k$.
 For  $ a,b \in \mathbb{R}^2 $ and $\xi,\eta \geq0 $, the following inequality holds:
 \begin{align}\label{inequality-01}
 \sum_{i=1}^2(a_i\xi+b_i\eta)^2=\|a\|_2^2\xi^2+\|b\|_2^2\eta^2+2<a,b>\xi\eta\leq(\|a\|_2\xi+\|b\|_2\eta)^2,
 \end{align}
 where the last inequality results form $|<a,b>|\leq\|a\|_2\|b\|_2$.
 Combining \eqref{up-x_Ok-norm} and \eqref{inequality-01}, we find
\begin{align}
\|(u^p - x_S )_{W}\|_2
&=\sqrt{\|(u^p - x_S )_{\Omega_k}\|_2^2+\|(u^p - x_S )_{(S\cup Q^p)\setminus\Omega_k}\|_2^2} \nonumber \\
&\leq C_4\|(u^p - x_S )_{(S\cup Q^p)\setminus\Omega_k}\|_2
+|1-\alpha|\cdot\|x^p - x_S \|_2
+\frac{\alpha\sqrt{1+\delta_{3k}}}{1-\delta_k} {\| \nu'\|}_2,\nonumber
\end{align}
where $C_4$ is given in Theorem \ref{thm-error-bound}. It follows from \eqref{up-x_COk-norm-2}  that
\begin{align} \label{up-x-norm}
\|(u^p - x_S )_{W}\|_2
\leq \left(|1-\alpha|+\sqrt{2}\delta_{3k}C_4\right)\cdot\|x^p-x_S\|_2
 +\left(\sqrt{2}C_4+\frac{\alpha}{1-\delta_k} \right)\sqrt{1+\delta_{3k}} \cdot\|\nu'\|_2.
\end{align}
Substituting \eqref{up-x-norm} into \eqref{xp1-x-up-xnorm-2}, we get
\begin{align} \label{xp1-x-norm}
\|x^{p+1} - x_S \|_2
\leq\rho\|x^p-x_S\|_2 +\left[\frac{\sqrt{1+\delta_{3k}}}{\sqrt{1-\delta^2_{2k}}}\cdot\left(2C_4 +\frac{\sqrt{2}\alpha}{1-\delta_k} \right) +\frac{\sqrt{1+\delta_{k}}}{1-\delta_{2k}}\right]\left\| \nu'\right\|_2
\leq\rho\|x^p-x_S\|_2 +C\left\| \nu'\right\|_2,
\end{align}
where $\rho,~C$ are given in \eqref{rho-C} and the last inequality is given by $\delta_{2k}\leq \delta_{3k}$. It means that \eqref{error-bound} holds.

{\bf Finally, we prove that $\rho<1$ under the conditions in  \eqref{RIC-bound}.}  Denote $F(r):=\frac{2r}{\sqrt{1-r^2}}\sqrt{\left(\frac{r}{1-r}\right)^2+1},~~r\in[0,1)$. Since $F(r)$ is strictly increasing on $[0,1)$, it follows from \eqref{Constants} and \eqref{RIC-bound} that
\begin{align}
\frac{C_3}{C_1}=F(\delta_{3k})\leq F(\sqrt{3/20})<1,\nonumber
\end{align}
where the last inequality is easily verified and thus omitted.  Hence, $C_3<C_1<1$, which implies that the second inequality in \eqref{RIC-bound} is well defined.
Since $\delta_k\leq\delta_{2k}\leq\delta_{3k}$, \eqref{rho-C} implies that
\begin{align}\label{rho-bound-1}
\rho\leq\frac{\sqrt{2}}{\sqrt{1-\delta^2_{3k}}}\left(|1-\alpha|+\sqrt{2}\delta_{3k}\sqrt{\left(\frac{\alpha\delta_{3k}}{1-\delta_{3k}}\right)^2+1}\right) =\frac{1}{C_1}\left(|1-\alpha|+\sqrt{2}\delta_{3k}\sqrt{\alpha^2C_2^2+1}\right),
\end{align}
where the equality is given by \eqref{Constants}.
It is easy to check that the following inequality
\begin{align}\label{inq-sqrt-linear}
\sqrt{\alpha^2t^2+1}\leq\left\{
\begin{array}{ll}
1+(\sqrt{1+t^2}-1)\alpha,~~~&0<\alpha\leq 1,\vspace{0.1cm}\\
t\alpha+\sqrt{1+t^2}-t,~~~&\alpha> 1,\\
\end{array}\right.
\end{align}
holds for any $t\geq 0$.  Combing \eqref{rho-bound-1} and \eqref{inq-sqrt-linear} with $t=C_2$ leads to
\begin{align}\label{rho-bound-2}
\rho&\leq\frac{1}{C_1}\cdot\left\{
\begin{array}{ll}
1-\alpha+\sqrt{2}\delta_{3k}\left[1+(\sqrt{1+C_2^2}-1)\alpha\right],~~~&0<\alpha\leq 1,\vspace{0.1cm}\\
\alpha-1+\sqrt{2}\delta_{3k}\left(C_2\alpha+\sqrt{1+C_2^2}-C_2\right),~~~&\alpha> 1,\\
\end{array}\right.\nonumber\\
&=\frac{1}{C_1}\cdot\left\{
\begin{array}{ll}
\left(1+\sqrt{2}\delta_{3k}-C_3\right)(1-\alpha)+C_3,~~~&0<\alpha\leq 1,\vspace{0.1cm}\\
(\sqrt{2}\delta_{3k}C_2+1)(\alpha-1)+C_3,~~~&\alpha> 1,\\
\end{array}\right.\nonumber
\end{align}
where the equality is given  by the definition of $C_3$ in \eqref{Constants}.
It follows from \eqref{RIC-bound} that $\rho<1$.
 ~~~~~ $\Box $

The global convergence of SDSNT follows directly from Theorem \ref{thm-error-bound} under certain conditions; see Remark \ref{rem-glob-converg}.
 \begin{Rem}\label{rem-glob-converg}
 If $x$ is a $k$-sparse vector and its measurement vector $y$ is accurate, i.e., $x_S=x$ and $\nu=0$, then \eqref{error-bound} reduces to $\|x^{p+1} - x\|_2 \leq \rho\|x^p - x \|_2$, where $\rho<1$ is given by \eqref{rho-C}.  Consequently,  the iterates $\{ x^p\}$ generated by SDSNT  converges to $x$.
 \end{Rem}

\section{Experiments}\label{sec-exper}
In the following experiments, we compare the proposed algorithm with several  state-of-the-art  algorithms, including NHTP \cite{ZXD21}, HTP \cite{F11}, NTP \cite{ZL22}, SP \cite{DM09},  and CoSaMP \cite{NT09}, in terms of  random sparse vector recovery and audio signal reconstruction. The stepsizes of SDSNT and  HTP  are set to $4$ and  $1$, respectively. The algorithm parameters of  NTP and NHTP are set identically to the original paper.
 Throughout this section, the problem dimension $n$ and the number of  measurements $m$ are fixed at $n=10000$ and $m=3000$. The measurement matrix $A\in\mathbb{R}^{m\times n}$ is  randomly generated, and its elements are sampled  independently from $\mathcal{N}(0, \frac{1}{m}).$

\subsection{Recovery of random sparse vector}
\vspace{-0.4cm}
\begin{figure}[htbp]
   \centering
  \subfigure[Success rate]{
\begin{minipage}[t]{0.45\linewidth}
  \includegraphics[interpolate=true, width=\textwidth,height=0.7\textwidth]{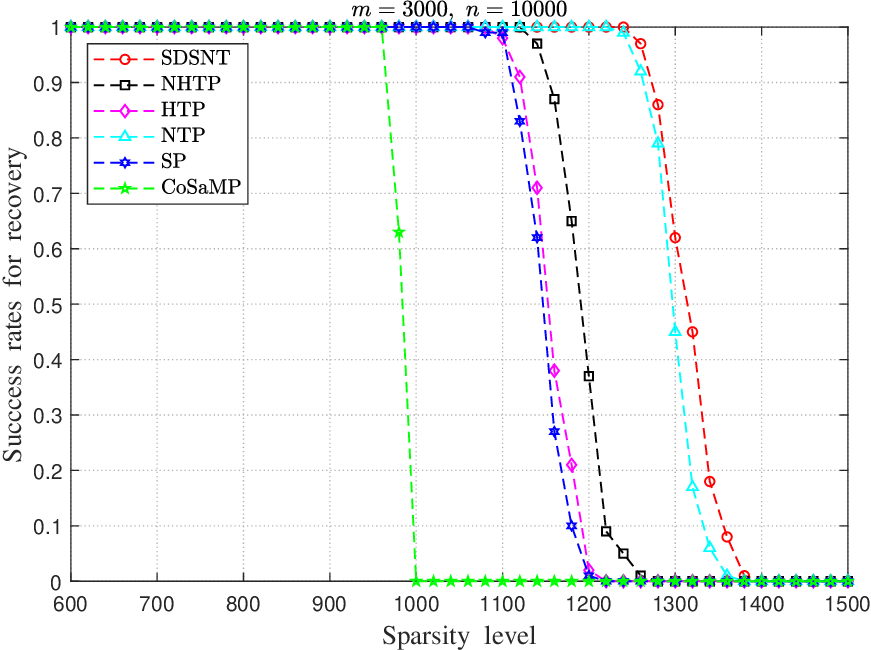}
   \end{minipage}} \hspace{0.2cm}
    \subfigure[Average CPU Time]{
\begin{minipage}[t]{0.45\linewidth}
  \includegraphics[interpolate=true,width=\textwidth,height=0.7\textwidth]{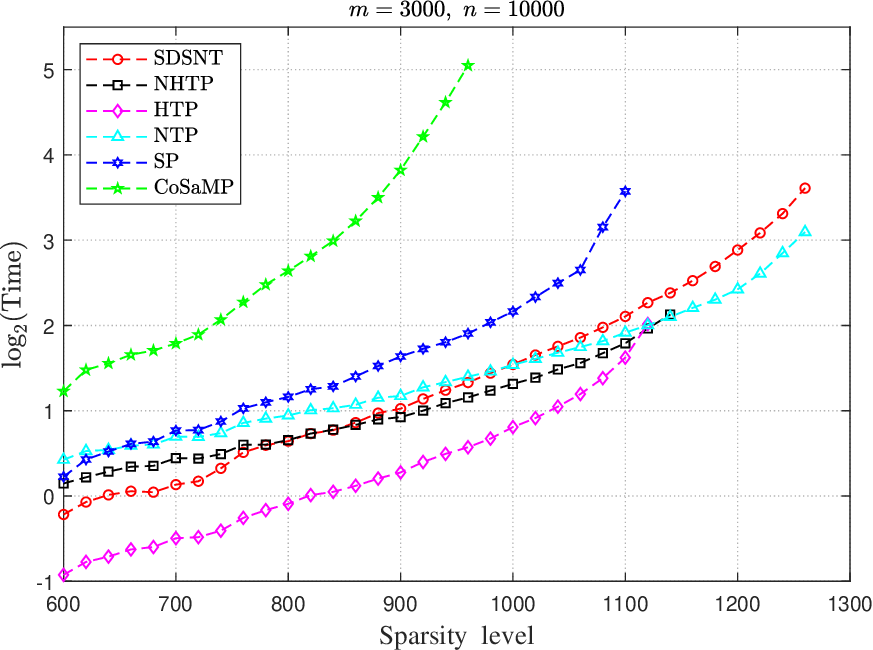}
   \end{minipage}
   }\vspace{-0.4cm}
   \caption{Success rates and average  CPU times of the algorithms}
   \label{Suc-Time}
   \end{figure}

We first evaluate  the recovery performance of the algorithms on 100 random instances for each sparsity level $k,$  ranging from 600 to 1500  with stepsize 20. For each  instance, the noisy measurements $y\in\mathbb{R}^{m}$ of a $k$-sparse vector $x^*\in\mathbb{R}^{n}$ are generated by $y=Ax^*+\epsilon h$, where $h\in\mathbb{R}^{m}$ is a normalized Gaussian noise vector and $\epsilon=0.01$. The nonzero positions of $x^*$ are uniformly distributed, with values drawn independently from $\mathcal{N}(0, 1).$ We aim to recover $x^*$ from the data $(y,A)$, and a recovery is deemed successful if the solution  $x^p$  produced by an algorithm satisfies $\|x^p-x^*\|_2/\|x^*\|_2\leq 10^{-3},$ which is also adopted as the stopping criterion. All algorithms are allowed up to 50 iterations. Fig. \ref{Suc-Time} illustrates the success rates and average runtimes of   all algorithms under comparison. Notably,  the average CPU time for successful recovery is recorded only for algorithms with a success rate of at least 0.9. Fig. \ref{Suc-Time}(a) indicates that SDSNT and NTP exhibit comparable recovery performance and both outperform the other four algorithms, with a marked advantage over  CoSaMP. On the other hand, Fig. \ref{Suc-Time}(b) shows that SDSNT is always faster than  CoSaMP and SP, but slower than HTP. For $k\leq840$, SDSNT exhibits lower computational cost than NHTP; however, when $k$  increases beyond this threshold, its average CPU time gradually surpasses  that of NHTP.  A similar pattern holds for NTP with  $k\leq 980$. The above results reveal that SDSNT achieves superior computational efficiency at low sparsity levels.

   \subsection{Reconstruction of audio signal}
In this subsection, we investigate the reconstruction of an audio signal $\widehat{z}\in\mathbb{R}^{ n}$ depicted in Fig. \ref{SNR-Time}(b). This signal is transformed into a compressible vector $\widehat{x}\in\mathbb{R}^{ n}$ via the orthogonal  transform $\mathbf{W}=\mathbf{W_2}\circ\mathbf{W_1}$, i.e., $\widehat{x}=\mathbf{W}\widehat{z}$. Here, $\mathbf{W_1}$ stands for  the  discrete cosine transform, and $\mathbf{W_2}$ denotes  the three-level discrete wavelet transform based on the  \emph{`sym12'} wavelet with periodic extension. Given measurements $y=(A\mathbf{W})\widehat{z}\in\mathbb{R}^{ m}$  of the signal $\widehat{z}$, we reconstruct a $k$-sparse vector $\widetilde{x}\in\mathbb{R}^{ n}$  from  the data $(y, A)$ using an algorithm with $k=m/3$, where $\widetilde{x}$ is the best $k$-term approximation of $\widehat{x}$. Consequently,  the reconstructed signal $\widetilde{z}$ is given by  $\widetilde{z}=\mathbf{W}^{-1}\widetilde{x}$.  We use the signal-to-noise ratio (SNR), defined by
$\mathrm{SNR}:=20 \log_{10}(\|\widehat{z}\|_2/\|\widehat{z}-\widetilde{z}\|_2),$
to evaluate the quality of the reconstructed signal. Each algorithm is allowed a maximum CPU time of 10 seconds. Fig. \ref{SNR-Time}(a) shows the SNR versus  the CPU time for all compared algorithms, with the per-iteration runtime recorded for each algorithm. The maximum SNR values attained by SDSNT, NHTP, HTP, NTP, SP and  CoSaMP
are approximately 26, 25, 25, 26, 23 and 14, respectively. Except for NTP, SDSNT achieves a maximum SNR at least 1 dB higher than that of any other algorithm. This superior SNR translates into better signal fidelity, as evidenced in  Fig. \ref{SNR-Time}(b), where the signal   reconstructed by SDSNT closely matches the original audio signal.
\vspace{-0.3cm}
\begin{figure}[htbp]
   \centering
   \subfigure[SNR vs. CPU time]{
\begin{minipage}[t]{0.45\linewidth}
  \includegraphics[interpolate=true, width=\textwidth,height=0.7\textwidth]{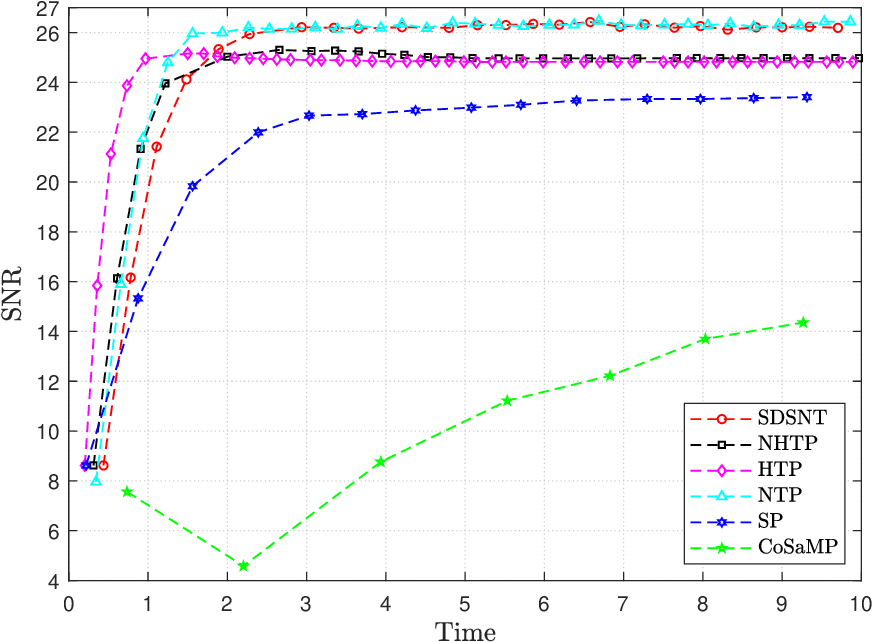}
   \end{minipage}}\hspace{0.2cm}
   \subfigure[Original audio signal vs. Reconstructed signal]{
\begin{minipage}[t]{0.45\linewidth}
  \includegraphics[interpolate=true,width=\textwidth,height=0.7\textwidth]{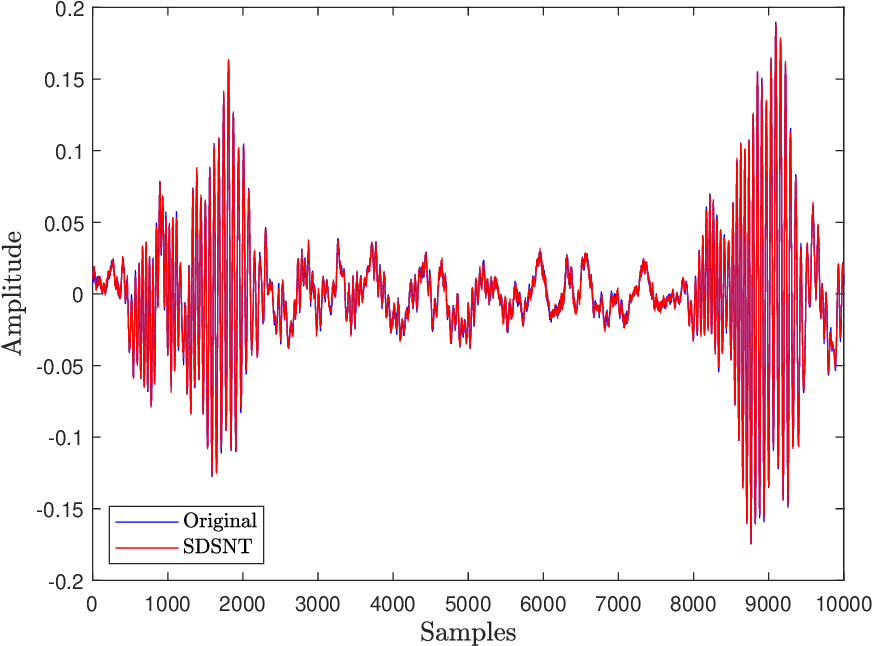}
   \end{minipage}
   }
\vspace{-0.4cm}
   \caption{(a) Comparison of SNR (in dB) vs. CPU time (in seconds); (b) The original audio signal  and the reconstructed one by SDSNT.}   \label{SNR-Time}
   \end{figure}

\section{Conclusions}
The SDSNT algorithm was proposed by combining a sparse descent direction with hard thresholding and orthogonal projection,  in which the direction is derived from the Newton direction within a low-dimensional subspace. This approach involves computing the principal submatrix of the Hessian that is positive definite under the RIP,  thereby circumventing the singularity of the Hessian. The global convergence of SDSNT has been established under the RIP  and with an appropriate stepsize. Numerical results show that SDSNT can compete with several state-of-the-art algorithms  in terms of both recovery capability and computational efficiency.

\section*{Data availability} The data will be openly available at the following URL after publication: {\href{https://github.com/zhongfengsun/zhongfengsun.github.io/tree/gh-pages/data}{https://github.com/zhongfengsun/zhongfengsun.github.io\\/tree/gh-pages/data}}.

\section*{Acknowledgments}
This work was supported by Shandong Provincial Natural Science Foundation (ZR2025MS99, ZR2023MA020), the National Natural Science Foundation of China (12371305, 12301400, 12471295, 12426306), Guangdong Basic and Applied Basic Research Foundation (2024A1515011566),  and Hetao
Shenzhen-Hong Kong Science and Technology Innovation Cooperation Zone Project (HZQSWS-KCCYB-2024016).

\end{document}